\documentclass[12pt,a4paper]{article}
\usepackage{amsmath,amssymb,amsthm, bm}
\usepackage{enumerate}
\usepackage[left=2cm,top=3cm,right=2cm,bottom=3cm]{geometry}
\usepackage{color}
\usepackage{relsize}%
\newtheorem{theorem}{Theorem}[section]
\newtheorem{corollary}[theorem]{Corollary}
\newtheorem{lemma}[theorem]{Lemma}
\newtheorem{remark}[theorem]{Remark}

\newtheorem{example}[theorem]{Example}
\newtheorem{definition}[theorem]{Definition}
\numberwithin{equation}{section}

\newcommand{\C}{{\ensuremath{\mathbb{C}}}}
\newcommand{\Cm}{{\ensuremath{\C^{m\times n}}}}

\newcommand{\Cnn}{{\ensuremath{\C^{n\times n}}}}
\newcommand{\Cmm}{{\ensuremath{\C^{m\times m}}}}

\newcommand{\Ra}{{\ensuremath{\cal R}}}
\newcommand{\Nu}{{\ensuremath{\cal N}}}
\newcommand{\rk}{{\ensuremath{\rm rank}}}

\newcommand{\rank}{{\ensuremath{\rm rank}}}

\begin{document}

\author{D.E. Ferreyra\thanks{Universidad Nacional de R\'io Cuarto, CONICET, FCEFQyN, RN 36 KM 601, 5800 R\'io Cuarto, C\'ordoba, Argentina. E-mail: \texttt{deferreyra@exa.unrc.edu.ar}} ~and
Saroj B. Malik\thanks{School of Liberal Studies, Ambedkar University, Kashmere Gate, Delhi, India. E-mail:  \texttt{saroj.malik@gmail.com}}}

\title{Relative EP matrices}
\date{}
\maketitle

\begin{abstract}
The purpose of the present work is to introduce the concept of  relative EP matrix of a rectangular matrix relative to a partial isometry (or, in short, $T$-EP matrix) hitherto unknown. We extend various  basic results on EP matrices and we study the relationship between $T$-hermitian, $T$-normal and $T$-EP matrices. The main theorems of this paper consist in  providing canonical forms of relative EP matrices when matrices involved are rectangular as well as square. We then use them  to characterize the relative EP matrices and show their properties. In fact, an interesting fact that has emerged is that  $A$ is $T$-EP if and only if  there is an EP matrix $C$ such that $A=CT$ and $C=TT^*C$ whatever be the matrix, square or rectangular. We also give various necessary and sufficient conditions for a matrix to be  $T$-EP.
\end{abstract}


\section{Introduction}

One of the problems in Matrix Theory is to  study and analyze the many classes of special type matrices for their properties (algebraic as well as geometric). The class of EP (range-hermitian) matrices is one such class. This class has attracted the attention of several authors since they were first defined  by H. Schwerdtfeger \cite{Sch}. Amongst many references available some of them are \cite{BaTr, CaMe,  ChTi, Dj, MoDjKo, Pearl}. One of the reasons for interest in this class is the fairly weaker condition on the matrix, namely the equality of it's range space and the range space of its conjugate transpose. EP matrices include in them the wide classes of matrices as special cases such as hermitian matrices, skew-hermitian matrices, unitary matrices,  normal matrices, and also  nonsingular matrices.

The EP matrices have been generalized in many ways. The first such generalization known as bi-EP matrices is given by Hartwig and Spindelb\"ock in \cite{HaSp}. Since then on several of their generalizations have appeared some of them being conjugate EP matrix \cite{MeIn}, $k$-EP \cite{MeKr}, weighted EP matrices \cite{Tian}, $m$-EP \cite{MaRuTh}, $k$-core EP \cite{FeLeTh3}, and $k$-DMP \cite{FeLeTh3}. All these extensions have been defined on the set of complex square matrices.  Our aim here in this work is to extend the class of EP matrices to a class of relative EP matrices when matrices are rectangular matrices.

We introduce in this paper the concept  of {\it relative EP matrix} for a rectangular matrix $A,$  extending the notion of square EP matrices to rectangular matrices. We recall that the concepts of relative hermitian and relative normal matrices were  introduced by Hestenes in \cite{Hes} in order to develop a spectral theory for rectangular matrices. Motivated by these definitions, we introduce in this paper the concept  of {\it relative EP matrix} for a rectangular matrix $A,$ (in short $T$-EP matrices, where $T$ is a fixed partial isometry). Like in classical case, the relative EP matrices are a generalization of each of relative hermitian and relative normal matrices. It is shown that every relative normal matrix is relative EP but not conversely. Several basic properties of EP matrices extend to relative EP matrices.

The main theorems of this paper are the development of canonical forms of relative EP matrices and using these  canonical forms we give various characterizations of these matrices.

The paper is organized as follows.  In Section 2, we introduce and characterize the concept of relative EP matrix as a generalization of the concept of EP matrix to the rectangular matrices. This new class of matrices is called the class of $T$-EP matrices and contains the class of relative normal matrices and therefore the class of relative hermitian matrices defined in \cite{Hes}. We also obtain here some properties of these matrices. Section 3 consists of obtaining canonical forms of a relative EP matrix where the matrix can be rectangular or square. The advantage of the canonical form in either case is that they reveal the relationship of a $T$-EP matrix $A$ to an EP matrix necessarily other than $A$ (when $T$ is not the trivial partial isometry $I_n$). Using these canonical forms we derive some characterizations of $T$-EP matrices. In particular,  we recover various  well-known results for EP matrices obtained by Pearl in \cite{Pearl}. Section 4 gives the interrelations between EP  and $T$-EP matrices as well as some more characterizations of $T$-EP matrices when $T$ belongs to certain matrix class. The final Section 5 deals with the problem of the sum of two relative EP matrices.

We denote the set of all $m \times n$ complex matrices by $\Cm$. For $A\in\Cm$, the symbols
$A^*$, $\rk(A)$,  $\Nu(A)$, and  $\Ra(A)$ will stand for the conjugate transpose, the rank, the null space, and the range space of $A$, respectively.
For $A\in\Cnn,~A^{-1}$ will denote the inverse of $A$ and $I_n$ will refer to the $n\times n$ identity matrix.

The symbol $A^\dag$  denotes the Moore-Penrose inverse of a matrix $A\in \Cm,$ which is the unique matrix satisfying the following four Penrose conditions
\cite{CaMe}:
\[ AA^\dag A=A, \quad A^\dag AA^\dag=A^\dag,
\quad (AA^\dag)^*=AA^\dag, \quad (A^\dag A)^*=A^\dag A.\]

A matrix $A\in \Cnn$ is an EP matrix if $\Ra(A)=\Ra(A^*)$, or equivalently $AA^\dag=A^\dag A$, see, e.g. \cite{CaMe, Pearl}.
Recall that $T\in \Cm$ is a partial isometry if it verifies $T=TT^*T$. Note that $T$ is a partial isometry if and only if $T^\dag=T^*$ \cite[Theorem 4.2.1]{CaMe}.
Throughout this work we will use  $T\in \Cm$ as a fixed partial isometry (We note that $T$ can also be a square matrix, but then it will be clear from the context of the usage).

\section{$T$-EP matrices}
In this section we introduce the concept of relative EP matrix for a matrix $A \in \Cm$ with respect to a partial isometry $T\in \Cm$. We present various equivalent characterizations of this new  class of matrices and we show that this class of matrices contains the following two classes of matrices defined by Hestenes \cite{Hes}:

\begin{definition}\label{R01} Let $A, T \in \Cm.$ Then $A$ is called relative hermitian to $T$ (or, in short, $T$-hermitian) if $A=TA^*T.$
\end{definition}

\begin{definition}\label{R02} Let $A,T\in \Cm.$ We say that $A$ is  relative normal to $T$ (or, in short, $T$-normal) if  $A=TT^*A=AT^*T$ and  $AA^*T=TA^*A$.
\end{definition}

Like in classical case as is known that every hermitian matrix is normal, we show that every  $T$-normal matrix  is $T$-hermitian.

\begin{theorem}\label{TH implies TN} Let $A,T\in \Cm.$ If $A$ is $T$-hermitian then $A$ is $T$-normal.
\end{theorem}
  \begin{proof} Since $A$ is $T$-hermitian, $A=TA^*T.$ So,
  \[AT^*T=(TA^*T)T^*T=TA^*(TT^*T)=TA^*T=A.\] Similarly, $TT^*A=TT^*(TA^*A)=A$. Also,
  \[AA^*T=(TA^*T)A^*T=TA^*(TA^*T)=TA^*A.\]
 Thus both (i) and (ii) in Definition \ref{R02} hold.
  \end{proof}

We now define the concept of relative EP matrix respect to a partial isometry $T$ as follows:

\begin{definition}\label{def TEP} Let $A,T\in \Cm.$  Then we say $A$ is  relative EP to $T$ (or, in short, $T$-EP) if $\Ra(A)=\Ra(TA^*T)$ and $A=AT^*T$.
\end{definition}

\begin{remark}{\rm In particular, when the partial isometry $T$ is the identity matrix  above definition reduces to $\Ra(A)=\Ra(A^*)$, which implies that $A$ is an EP matrix \cite{Pearl}.}
\end{remark}

We next show that a $T$-normal matrix is a $T$-EP matrix.

\begin{theorem}\label{TN implies TEP} Let $A,T\in \Cm$ If $A$ is $T$-normal then $A$ is $T$-EP.
 \end{theorem}
\begin{proof} Let $A$ be $T$-normal. Then $A=AT^*T=TT^*A$ and $AA^*T=TA^*A$ (or equivalently $T^*AA^*=A^*A T^*$). Clearly, $A=AT^*T$. It remains to see that $\Ra(A)=\Ra(TA^*T)$. In fact,
\begin{align*}
\Ra(A)&=\Ra(AA^*)\\
&=\Ra(TT^*AA^*) \\
&= \Ra(TA^*AT^*) \\
&= \Ra(TA^*TT^*AT^*) \\
&= \Ra(TA^*T(TA^*T)^*) \\
&= \Ra(TA^*T).
\end{align*}
\end{proof}
We now give a few examples.
Just as an observation, if the matrices involved are square matrices, every EP matrix $A\in \Cnn$ is relative EP, relative to trivial partial isometry $I_n$.

\begin{example}A trivial example of a $T$-EP matrix is the partial isometry $T$ itself because $T=TT^*T$.\end{example}

\begin{example}\label{E1}Let $T= \begin{bmatrix}
          0 & 1 & 0 \\
          0 & 0 & 0 \\
          0 & 0 & 1 \\
        \end{bmatrix},$ then $T$ is a partial isometry and all  $T$-hermitian matrices are of the type
$A= \begin{bmatrix}
                                                                                          0 & a & z \\
                                                                                          0 & 0 & 0 \\
                                                                                          0 & \overline{z} & b \\
                                                                                        \end{bmatrix},$ where $a,b$ are real numbers and $z\in \C$.
Note that $A$ is not hermitian. Moreover,  $A$ is not EP.
\end{example}
\begin{example}\label{E2}Let $T$ be the partial isometry in Example \ref{E1}.
Let $A=\begin{bmatrix}
0 & 1 & 1 \\
0 & 0 & 0 \\
0 & 1 & 1 \\
\end{bmatrix}.$
By using previous example, it is clear that $A$ is $T$-hermitian.
Therefore, by Theorems \ref{TH implies TN}  and \ref{TN implies TEP} we have that $A$ is $T$-normal and $T$-EP. However, note that $A$ is not EP.
\end{example}
Our next example shows that the converse of Theorem \ref{TN implies TEP} is false.

\begin{example}\label{E3} Let $T=\begin{bmatrix}
0 & 1 & 0 \\
0 & 0 & 1 \\
0 & 0 & 0 \\
\end{bmatrix}.$  Then $T$ is a partial isometry. Let $A=\begin{bmatrix}
0 & 1 & 1 \\
0 & 0 & 1 \\
0 & 0 & 0 \\
\end{bmatrix}.$ Then $A$ is $T$-EP, but not $T$-normal.
\end{example}

The following two lemmas are fundamental to developing the properties and characterizations of relative EP matrices. We also need to use the following important relation between the null space of a matrix $A\in \Cm$  and the range space of its  conjugate transpose $A^*$, namely
\begin{equation} \label{eqA}
\Ra(A)=\Nu(A^*)^\perp,
\end{equation}
where the superscript $\perp$ denotes the orthogonal complement of $\Nu(A^*)$.

\begin{lemma} \label{lemma 1} Let $A,T \in \Cm.$   Then the following are equivalent:
\begin{enumerate}[{\rm (i)}]
\item  $A=TT^*A$;
\item $\Ra(A)\subseteq \Ra(T)$;
\item $\Nu(T^*)\subseteq \Nu(A^*)$;
\item $A^\dag=A^\dag TT^*$.
\end{enumerate}
\end{lemma}
\begin{proof}
(i)$\Rightarrow$ (ii). Easy. \\
(ii)$\Rightarrow$ (iii). Using \eqref{eqA} in  $\Ra(A)\subseteq \Ra(T)$
 implies $\Nu(A^*)^\perp\subseteq \Nu(T^*)^\perp$, which yields to $\Nu(T^*)\subseteq \Nu(A^*)$.
 \\
(iii)$\Rightarrow$ (iv). As $\Nu(A^*)=\Nu(A^\dag)$ we have $\Nu(T^*)\subseteq \Nu(A^\dag)$ and is further equivalent to $A^\dag=A^\dag (T^*)^\dag T^*$.  As $T^\dag=T^*,$ we have $A^\dag=A^\dag (T^*)^*T^*=A^\dag T T^*$. \\
(iv)$\Rightarrow$ (i). Pre-multiplying $A^\dag=A^\dag TT^*$ by $A^*A,$ we obtain $A^*=A^* TT^*$ and therefore, $A=TT^*A$ by taking $*$ on both sides.

\end{proof}

Similarly one can show the following:

\begin{lemma}\label{lemma 2} Let $A,T \in \Cm.$  Then the following are equivalent:
\begin{enumerate}[{\rm (i)}]
\item $A=AT^*T$;
\item $\Nu(T) \subseteq \Nu(A)$;
\item $\Ra(A^*)\subseteq \Ra(T^*)$;
\item $A^\dag= T^*TA^\dag$.
\end{enumerate}
\end{lemma}

We now use these lemmas to obtain some characterizations and properties of  $T$-EP matrices.

\begin{theorem}\label{characterization 1} Let $A,T\in \Cm$. Then the following are equivalent:
\begin{enumerate}[{\rm (i)}]
\item $A$ is $T$-EP (i.e., $\Ra(A)=\Ra(TA^*T)$ and $A=AT^*T$);
\item $\Ra(A)=\Ra(TA^*T)$ and $\Nu(T)\subseteq \Nu(A)$;
\item $\Nu(A^*)=\Nu(T^*AT^*)$ and $\Ra(A^*)\subseteq \Ra(T^*)$;
\item $\Ra(A)=\Ra(TA^*)$ and $\Nu(T)\subseteq \Nu(A)$;
\item $\Ra(A)=\Ra(TA^*)$ and $A=AT^*T$;
\item $\Nu(A^*)=\Nu(AT^*)$ and $\Ra(A^*)\subseteq \Ra(T^*)$.
\end{enumerate}
\end{theorem}
\begin{proof}
(i) $\Leftrightarrow$ (ii). It follows from Definition \ref{def TEP} and Lemma \ref{lemma 1}. \\
(ii) $\Leftrightarrow$ (iii). It is a  direct consequence of \eqref{eqA}. \\
(ii) $\Rightarrow$ (iv).  We only need to prove that $\Ra(A)=\Ra(TA^*)$.   Since $\Ra(A)=\Ra(TA^*T)$ we obtain $\Ra(TA^*T)\subseteq \Ra(TA^*)$ and $\rk(A)=\rk(TA^*T)$. Moreover,  $\rk(TA^*)\le \rk(A^*)=\rk (A) =\rk(TA^*T)$. Therefore, $\Ra(A)=\Ra(TA^*)$. \\
(iv) $\Rightarrow$ (ii). Suppose $\Ra(A)=\Ra(TA^*)$ and $\Nu(T)\subseteq \Nu(A)$. Only we need to show that $\Ra(A)=\Ra(TA^*T)$. In fact, clearly $\Ra(TA^*T)\subseteq \Ra(A)$. Also, as $\Ra(A)\subseteq \Ra(T)$ implies $TT^*A=A$ by Lemma \ref{lemma 1} (i). So, $\rk(A)=\rk(TA^*)=\rk(TA^*TT^*)\le \rk(TA^*T)$, whence   $\Ra(A)=\Ra(TA^*T)$. \\
(iv) $\Leftrightarrow$ (v) By Lemma \ref{lemma 2}. \\
(v) $\Leftrightarrow$ (vi) Follows from \eqref{eqA}  and Lemma \ref{lemma 2}.
\end{proof}

\begin{theorem}\label{properties 1} Let $A\in \Cm$ be a $T$-EP matrix, then the following hold:
\begin{enumerate}[{\rm (i)}]
\item $A=TT^*A$.
\item $\rk(A)=\rk(TA^*T)=\rk(TA^*)=\rk(A^*T)$;
\item $(T^*A)^\dag=A^\dag  T$;
\item $(AT^*)^\dag= T A^\dag$;
\item $(T^*AT^*)^\dag=TA^\dag T$.
\end{enumerate}
\end{theorem}
\begin{proof}
(i) As $A$ is $T$-EP, we have $\Ra(A)=\Ra(TA^*T)$, and so $\Ra(A)\subseteq \Ra(T)$. Now, Lemma \ref{lemma 1} implies $A=TT^*A$.\\
(ii) By definition of $T$-EP matrix it is clear that $\rk(A)=\rk(TA^*T)$.   Thus,  \[\rk(A)=\rk(TA^*T)\le \rk(TA^*)\le \rk(A)=\rk(TA^*T)\le \rk(A^*T)\le \rk(A),\]
whence $\rk(A)=\rk(TA^*T)=\rk(TA^*)=\rk(A^*T)$.\\
(iii) We check this by direct verification. Let $X:=A^\dag T $. By definition of $T$-EP matrix and part (i) we have   $A=AT^*T$ and $A=TT^*A$, respectively. Thus,
\[T^*AXT^*A=T^*A A^\dag T T^*A = T^*A A^\dag A=T^*A,\]
\[XT^*AX=A^\dag T T^*A A^\dag T= A^\dag AA^\dag T^*=A^\dag T^*=X,\]
\[(T^*AX)^*=(T^*AA^\dag T)^*=T^*(AA^\dag)^* T=T^*AA^\dag T=T^*AX,\]
\[(XT^*A)^*=(A^\dag T T^*A)^*=(A^\dag A)^*=A^\dag A=A^\dag TT^*A= XT^*A.\]
By uniqueness of the Moore-Penrose inverse we have $X=(T^*A)^\dag$. \\
Items (iv) and (v) can be proved along lines of (iii).
\end{proof}

\section{Canonical forms of $T$-EP matrices and consequences}
 In this section we exhibit the main results of this paper namely, canonical forms of a $T$-EP matrix $A$.
 We first derive a canonical form of a  $T$-EP matrix, when the matrix is rectangular. The tool we use is the Singular Value Decomposition (SVD). As a special case we have a canonical form of a square matrix. However, we provide another canonical form of a square matrix by using the Hartwig-Spindelb\"ock decomposition of a square matrix, which by itself was derived by using SVD. The reason to include this is that it has its own merits.  We then discuss immediate consequences of both the canonical forms. We also give some characterizations of $T$-EP matrices.

\begin{theorem}(SVD)\label{svd} Let $A\in \Cm$ a nonnull matrix of rank $r>0$ and let $\sigma_1\ge\sigma_2\ge\cdots \ge \sigma_r>0$ be the singular values of $A$.  Then there exist unitary matrices $U\in \Cmm$ and $V\in \Cnn$ such that
$A=U\begin{bmatrix}                                                                              \Sigma & 0 \\
0 & 0
\end{bmatrix}V^*$, where $\Sigma=diag(\sigma_1, \sigma_2, \hdots, \sigma_r)$. In particular, the Moore-Penrose inverse of $A$ is given by
\begin{equation}\label{MP respecto to SVD}
A^\dag=V\begin{bmatrix}                                                                              \Sigma^{-1} & 0 \\
0 & 0
\end{bmatrix}U^*.
\end{equation}
\end{theorem}

\begin{theorem}\label{canonical form rect} Let $A, T\in \Cm$ be such that $\rank(A)=r$. Then the following are equivalent.
\begin{enumerate}[{\rm(i)}]
\item  $A$ is $T$-EP;
\item  There exists nonsingular matrix $D$ of order $r$ and a unitary matrix $U$ such that
\begin{equation}\label{A respect to T}
 A= U\begin{bmatrix}
               D & 0 \\
               0 & 0 \\
              \end{bmatrix} U^*T \quad \text{and} \quad T=U\begin{bmatrix}
         T_1 & 0 \\
          0 & T_4 \\
        \end{bmatrix}V^*,
\end{equation}
where $T_1$ and $T_4$ are matrices satisfying $T_1T^*_1=I_r$ and $T_4T^*_4T_4=T_4$.
\item There exists an EP matrix $E\in \Cmm$ such that $A=ET$ and $TT^*E=E$.
\end{enumerate}
\end{theorem}
\begin{proof}
(i) $\Rightarrow $ (ii).
Let $A$ be $T$-EP. By Theorem \ref{svd}, there exist unitary matrices $U\in \Cmm$ and $V\in \Cnn$, and a nonsingular  matrix $\Sigma$ such that $A=U\begin{bmatrix}
 \Sigma & 0 \\
  0 & 0 \\
 \end{bmatrix}V^*.$
Since $A$ is $T$-EP,  Theorem \ref{characterization 1} (v) implies $\Ra(A)=\Ra(TA^*)$ and $A=AT^*T$. Let $T$ be partitioned as
$T=U \begin{bmatrix}
 T_1 & T_2 \\
 T_3 & T_4
 \end{bmatrix}V^*$,
such that the product $TA^*$ and $AT^*$ are defined.
Clearly,  $A=AT^*T$ is equivalent to
\[\begin{bmatrix}
 \Sigma & 0 \\
 0 & 0
 \end{bmatrix}= \begin{bmatrix}
                \Sigma & 0 \\
                0 & 0
              \end{bmatrix}\begin{bmatrix}
T_1^*T_1+T_3^*T_3 & T_1^*T_2+T_3^*T_4\\
T_2^* T_1+T_4^*T_3 & T_2^* T_2+T_4^* T_4
 \end{bmatrix},\]
equivalently,
\begin{equation}\label{eq2}
T^*_1T_1+T^*_3T_3=I_r,
\end{equation}
\begin{equation}\label{eq3}
T^*_1T_2+T^*_3T_4=0.
\end{equation}
Since $T$ is a partial isometry, i.e., $T=TT^*T$, using \eqref{eq2} and \eqref{eq3} we obtain
\[\begin{bmatrix}
 T_1 & T_2 \\
 T_3 & T_4 \\
 \end{bmatrix}=\begin{bmatrix}
 T_1 & T_2 \\
 T_3 & T_4 \\
 \end{bmatrix}\begin{bmatrix}
                I_r & 0 \\
                0 & T^*_2T_2+T_4^*T_4
              \end{bmatrix},\]
 whence
 \begin{equation}\label{eq4}
 T_2=T_2T^*_2T_2+T_2T^*_4T_4,
 \end{equation}
\begin{equation}\label{eq5}
 T_4=T_4T^*_2T_2+T_4T^*_4T_4.
 \end{equation}
Also, as $\Ra(A)=\Ra(TA^*)$ we get $AA^\dag TA^*=TA^*$. By using \eqref{MP respecto to SVD}, after simple matrix computations we obtain
 \[\begin{bmatrix}
   T_1\Sigma^* & 0 \\
   0 & 0
 \end{bmatrix}=\begin{bmatrix}
                T_1 \Sigma ^* & 0 \\
                T_3 \Sigma^* & 0
              \end{bmatrix},\]
whence $T_3=0$ because  $\Sigma$ is nonsingular. In consequence,  from \eqref{eq2} we obtain $T_1^* T_1=I_r$ or equivalently $T_1T_1^*=I_r$ as $T_1$ is a square matrix of order $r$. Therefore, \eqref{eq3} implies $T_2=0$. So, from \eqref{eq5} we get $T_4=T_4T^*_4T_4$ i.e., $T_4$ is a partial isometry. \\
Finally, once again from $A=AT^*T$ we have
   \[A=U\begin{bmatrix}
         \Sigma T^*_1 & 0 \\
         0 & 0 \\
       \end{bmatrix}U^*T=U\begin{bmatrix}
                           D & 0 \\
                           0 & 0
                         \end{bmatrix}U^*T,\] where  $D$ is the matrix $\Sigma T^*_1$ which clearly is nonsingular. \\
(ii) $\Rightarrow $ (iii).
 Write $E=U\begin{bmatrix}
               D & 0 \\
               0 & 0
              \end{bmatrix}U^*$, where $D$ is nonsingular. Clearly $E$ is EP and $A=ET$.
Also, since  $T_1T^*_1=I_r$, it follows that
\[TT^*E=U\begin{bmatrix}
 I_r & 0 \\
 0 & T_4 T_4^*
 \end{bmatrix}\begin{bmatrix}
               D & 0 \\
               0 & 0
              \end{bmatrix}U^*=U\begin{bmatrix}
               D & 0 \\
               0 & 0
              \end{bmatrix}U^*=E.\]
(iii) $\Rightarrow $ (i).
  We assume $A=ET$ and $TT^*E=E$.
As $E$ is EP then $\Ra(E)=\Ra(E^*).$ Thus,
\[\Ra(TA^*)=\Ra(TT^*E^*)=TT^*\Ra(E^*)=TT^*\Ra(E)=\Ra(TT^*E)=\Ra(E),\]
whence
$\Ra(A)=\Ra(ET) \subseteq \Ra(E)=\Ra(TA^*)$. \\ Moreover, $\rank( AT^*)\leq \rank(A).$  Thus,
$\Ra(A)=\Ra(TA^*)$. Also, as $A=ET$ we get $\Nu(T) \subseteq \Nu(ET)=\Nu(A)$. Thus, by Theorem 2.14 (iv) $A$ is $T$-EP.
\end{proof}

\begin{remark}{\rm \label{remark TT*E=E} Note  that under any one of equivalent conditions (i)-(iii) in Theorem \ref{canonical form rect} we have $ETT^*=E$.}
\end{remark}

Theorem  \ref{canonical form rect} (iii) enables us to obtain the following formula for the Moore-Penrose inverse of a matrix.

\begin{theorem} Let $A$, $T$, and $E$ be as in Theorem \ref{canonical form rect}. Then the Moore-Penrose inverse of $A$ is given by
\begin{equation}\label{MP of ET}
A^\dag=T^* E^\dag.
\end{equation}
\end{theorem}
\begin{proof}
We check this by direct verification. Let $X:=T^*E^\dag$. As is $T$-EP, Theorem \ref{canonical form rect} and Remark \ref{remark TT*E=E} imply  $A=ET$ and  $E=TT^*E=ETT^*$. Therefore,
\[AXA=ET(T^*E^\dag)ET = (ETT^*)E^\dag ET=ET=A,\]
\[XAX=T^*E^\dag (ET T^*) E^\dag = T^*E^\dag E E^\dag=T^* E^\dag=X.\]
Also we have $AX=(ETT^*)E^\dag=EE^\dag$ and $XA=T^*E^\dag ET$. Clearly, $AX$ and $XA$ are hermitian matrices. So, by uniqueness of the Moore-Penrose inverse we have $X=A^\dag$.
\end{proof}

Some important  consequences of Theorem \ref{canonical form rect} are the following:

\begin{corollary} Let $A$, $T$, $D$, and $U$ be as in Theorem \ref{canonical form rect}. Then the Moore-Penrose inverse of $A$ is given by
\begin{equation}\label{MP respect to T}
A^\dag=T^* U\begin{bmatrix}
               D^{-1} & 0 \\
               0 & 0
              \end{bmatrix} U^*.
\end{equation}
\end{corollary}
\begin{proof}
It directly follows from \eqref{MP of ET} by using the fact that  $E=U\begin{bmatrix}
               D & 0 \\
               0 & 0
              \end{bmatrix} U^*$.
\end{proof}

\begin{corollary}  Let $A, T\in \Cm$ be such that $\rank(A)=r$.  If $A$ is $T$-EP, then there exists a partial isometry $R \in \Cnn$ and  a nonsingular matrix $D$ of order $r$ such that
\begin{equation}\label{T*A}
T^*A= R\begin{bmatrix}
               D & 0 \\
               0 & 0
              \end{bmatrix} R^*.
\end{equation}
\end{corollary}
\begin{proof} Let $A$ be $T$-EP. By Theorem \ref{canonical form rect}, there exists a nonsingular matrix $D$ of order $r$ and a unitary matrix $U$ such that
$A= U\begin{bmatrix}
               D & 0 \\
               0 & 0
              \end{bmatrix} U^*T$.
Pre-multiplying the previous equality by $T^*$ and  by taking $R:=T^*U$ we obtain  \eqref{T*A}. Clearly, $R$ is
a partial isometry.
\end{proof}

\begin{theorem}\label{characterization 2} Let $A, T\in \Cm$.  Then $A$ is $T$-EP  if and only if  $TA^\dag A=AA^\dag T$ and $A=AT^*T$.
\end{theorem}
\begin{proof} In order to prove the necessity of the condition it is enough to show that $TA^\dag A=AA^\dag T$ because  clearly $A=AT^*T$ holds by definition of $T$-EP matrix.   Thus, by  Theorem \ref{canonical form rect} and Remark \ref{remark TT*E=E}, there exists an EP matrix $E$ such that  $A=ET$ and $E=TT^*E=ETT^*$. Moreover, \eqref{MP of ET} implies $A^\dag=T^*E^\dag$. Consequently,  as $EE^\dag=E^\dag E$ we obtain
\[A A^\dag  T= ETT^*E^\dag T=EE^\dag T=TT^*EE^\dag T=TT^*E^\dag ET=TA^\dag A.\]
Conversely, we assume that $TA^\dag A=AA^\dag T$ and $A=AT^*T$ hold. Consider $TA^\dag A=AA^\dag T.$ Multiplying on right by $A^*,$ we have $TA^*=AA^\dag T A^*.$ This implies $\Ra(TA^*)\subseteq \Ra(A)$. Since $\rank(A)=\rank(AT^*T)\le \rank(AT^*)=\rank((AT^*)^*) = \rank(TA^*)$,  we get
$\Ra(A)=\Ra(TA^*).$ Now, from Theorem \ref{characterization 1} (v) it follows that $A$ is $T$-EP.
\end{proof}

\begin{corollary}\label{characterization 3} Let $A, T\in \Cm$.  Then $A$ is $T$-EP  if and only if  $TA^\dag A=AA^\dag T$ and any one of equivalent conditions (i)-(iv) in Lemma \ref{lemma 2} hold.
\end{corollary}

\begin{theorem}\label{characterization 4} Let $A, T\in \Cm$. Then the following are equivalent:
\begin{enumerate}[\rm(i)]
\item $A$ is $T$-EP;
\item  $A^*$ is $T^*$-EP;
\item $A^\dag$ is $T^*$-EP.
\end{enumerate}
\end{theorem}
\begin{proof}  (i)$\Rightarrow$(ii). By Theorem \ref{characterization 2}, it is enough to establish
$T^*(A^*)^\dag A^*=A^*(A^*)^\dag T^*$ and $A^*=A^* (T^*)^* T^*=A^*TT^*$. As $A$ is $T$-EP, from Theorem \ref{characterization 2}  we have $AA^\dag T=TA^\dag A$ and $A=AT^*T$. Moreover, Theorem \ref{properties 1} (i) implies $A= TT^* A$, and by taking $*$ on both sides we have  $A^*=A^* T T^*$. It remains to prove that $T^*(A^*)^\dag A^*=A^*(A^*)^\dag T^*$ but this equality is true if and only if $AA^\dag T=TA^\dag A$. This shows (ii).\\
(ii)$\Rightarrow$(i). It is similar to the previous proof by interchanging the roles of $A$ and $A^*$. \\
(i)$\Rightarrow$(iii). By Theorem \ref{characterization 2}, we know that $A$ is $T$-EP  implies  $AA^\dag T=TA^\dag A$ and $A=A T^*T$.  Moreover, Theorem \ref{properties 1} (i) implies $A= TT^* A$, which is equivalent to $A^\dag=A^\dag TT^*$ by Lemma \ref{lemma 1}.
As $TA^\dag A=AA^\dag T,$ taking $*$ on both sides, we have $A^\dag A T^*= T^* AA^\dag.$ Using  $A=(A^\dag)^\dag$  we have $A^\dag (A^\dag)^\dag T^*= T^* (A^\dag)^\dag A^\dag$. So, once again  Theorem \ref{characterization 2} implies that $A^\dag$ is $T^*$-EP.\\
(iii)$\Rightarrow$(i). Along lines of (ii)$\Rightarrow$(i).
 \end{proof}

\begin{theorem}\label{characterization 1 bis} Let $A,T \in \Cm$. Then the following are equivalent:
\begin{enumerate}[\rm (i)]
\item $A$ is $T$-EP;
\item $AT^*$ is EP and $A=AT^*T$;
\item $TA^*$ is EP and $A=AT^*T$;
\item $TA^\dag$ is EP and $A=AT^*T$.
\end{enumerate}
\end{theorem}
\begin{proof} (i)$\Rightarrow$(ii). We assume  $A$ is $T$-EP. Then by  Theorem \ref{canonical form rect} and Remark \ref{remark TT*E=E}, there exists an EP matrix $E$ such that $A=ET$ and $E=ETT^*=TT^*E$. Therefore $AT^*=ETT^*=E$, that is, $AT^*$ is EP. Clearly, $A=AT^*T$ by definition of $T$-EP matrix. \\
(ii) $\Rightarrow$ (i). We consider  $AT^*$ is EP and $A=AT^*T$.  Taking $E=AT^*$, from $A=AT^*T$  we have $A=ET$, where $E$ is clearly an EP matrix. Moreover, as $\Ra(E)=\Ra(E^*)$ we have $\Ra(E)\subseteq \Ra(T)$, which is equivalent to $TT^*E=E$. Now, by using Theorem \ref{canonical form rect} it follows that $A$ is $T$-EP.\\
(ii) $\Leftrightarrow$ (iii). It is a direct consequence of the fact that a square complex matrix $B$ is EP if and only if $B^*$ is EP. \\
(i) $\Rightarrow$ (iv). We assume  $A$ is $T$-EP. By Theorem \ref{properties 2} (v) we have $(AT^*)^\dag=TA^\dag$. Since a matrix $B$ is EP if and only  $B^\dag$ is EP, we have the equivalence.
\end{proof}

Some more properties of $T$-EP matrices are established in the following theorem:

\begin{theorem}\label{properties 2} Let $A\in \Cm$ be a $T$-EP matrix, then the following hold:
\begin{enumerate}[{\rm (i)}]
\item $(TA^*)^\dag=(A^*)^\dag  T^*$;
\item $(A^*T)^\dag= T^* (A^*)^\dag$;
\item $(TA^*T)^\dag=T^*(A^*)^\dag T^*$;
\item $(TA^\dag)^\dag=A  T^*$;
\item $(A^\dag T)^\dag= T^* A$;
\item $(TA^\dag T)^\dag=T^*A T^*$.
\end{enumerate}
\end{theorem}
\begin{proof} It is a direct consequence from Theorem \ref{properties 1} and Theorem \ref{characterization 4}.
\end{proof}

We note that for square matrices we can obtain a canonical form from Theorem \ref{canonical form rect} by simply taking $m=n.$ However, we present another canonical form a square matrix $A$ for which we use the famous Hartwig-Spindelb\"ock decomposition \cite{HaSp}(which it self was derived from SVD of $A$). For this , we assume all matrices are square from now onwards and for the next section of this paper. To give this canonical form  we need the following two results.

\begin{theorem}\cite[Hartwig-Spindelb\"ock decomposition]{HaSp} Let $A\in \Cnn$ of rank $r>0$. Then there exists a unitary matrix $U\in \Cnn$ such that
\begin{equation} \label{HS}
 A = U\left[\begin{array}{cc}
\Sigma K & \Sigma L \\ 0 & 0 \end{array}\right]U^*, \end{equation} where $\Sigma=\text{diag} (\sigma_1 I_{r_1},\sigma_2 I_{r_2},\dots, \sigma_t I_{r_t})$ is the
diagonal matrix of singular values of $A$, $\sigma_1>\sigma_2>\cdots>\sigma_t>0$, $r_1+r_2+\cdots+r_t=r$, and $K\in \C^{r\times r}$, $L\in \C^{r\times(n-r)}$
satisfy $KK^*+LL^*=I_r$. \end{theorem}

\begin{theorem} \label{HS representation} \cite{BaTr} Let $A\in \Cnn$ be a matrix  written as in (\ref{HS}). Then
\[A^\dag = U\left[\begin{array}{cc} K^* \Sigma^{-1} & 0 \\ L^* \Sigma^{-1} & 0
\end{array}\right]U^*, \quad AA^\dag = U\left[\begin{array}{cc} I_r & 0\\ 0 & 0
    \end{array}\right]U^*.\]
\end{theorem}

\begin{theorem}\label{canonical form} Let $A,T \in \Cnn,$  $T$ a partial isometry, and $r=\rank(A)$. Then the following are equivalent.
\begin{enumerate}[{\rm(i)}]
\item  $A$ is $T$-EP;
\item  There exists nonsingular matrix $D$ of order $r$ and a unitary matrix $U$ such that
           \[A= U\begin{bmatrix}
               D & 0 \\
               0 & 0 \\
              \end{bmatrix} U^*T;\]
where $T=U\begin{bmatrix}
         T_1 & T_2 \\
          T_3 & T_4 \\
        \end{bmatrix}U^*$ with $T_1T^*_1+T_2T^*_2=I_r$ and $T_3T^*_1+T_4T^*_2=0$.
\item There exists an EP matrix $C$ such that $A=CT$ and $TT^*C=C$.
\end{enumerate}
In particular, under any one of equivalent conditions (i)-(iii) we have $CTT^*=C$.
\end{theorem}
\begin{proof} We show (i) $\Rightarrow $ (ii) $\Rightarrow $ (iii) $\Rightarrow $ (i). \\
(i) $\Rightarrow $ (ii).
Let $A$ be $T$-EP. By Theorem \ref{characterization 1} (v) $\Ra(A)=\Ra(TA^*)$ and $A=AT^*T.$
We write $A$ as in \eqref{HS}. Let $T$ be partitioned in conformation with the partition of $A$ as $T=U\begin{bmatrix}
     T_1 & T_2 \\
     T_3 & T_4 \\
   \end{bmatrix}U^*.$
Then $A=AT^*T$ is equivalent to
\[
\begin{bmatrix}
         \Sigma K & \Sigma L \\
          0 & 0 \\
        \end{bmatrix}
        =\begin{bmatrix}
          \Sigma K & \Sigma L \\
          0 & 0 \\
        \end{bmatrix}
        \begin{bmatrix}
          T^*_1T_1+T_3^*T_3 & T^*_1 T_2 + T_3^*T_4  \\
          T^*_2 T_1+T_4^*T_3 & T_2^* T_2+T^*_4T_4 \\
        \end{bmatrix},
\]
or equivalently
\begin{equation}\label{eq11}
\Sigma K=\Sigma K(T^*_1T_1+T^*_3T_3)+ \Sigma L (T^*_2 T_1+T_4^*T_3),
\end{equation}
\begin{equation}\label{eq12}
\Sigma L=\Sigma K(T^*_1 T_2 + T_3^*T_4)+\Sigma L (T_2^* T_2+T^*_4T_4).
\end{equation}
Also, by Theorem \ref{properties 1} (i) we have that $A=TT^*A$, which in turn  is equivalent to
\begin{equation}\label{eq13}
\Sigma K=(T_1T^*_1+T_2T^*_2)\Sigma K
\end{equation}
\begin{equation}\label{eq14}
\Sigma L=(T_1T^*_1+T_2T^*_2)\Sigma L
\end{equation}
\begin{equation}\label{eq15}
0=(T_3T^*_1+T_4T^*_2)\Sigma K
\end{equation}
\begin{equation}\label{eq16}
0=(T_3T^*_1+T_4T^*_2)\Sigma L
\end{equation}
Multiplying \eqref{eq13} by $K^*$ on right and \eqref{eq14} by $L^*$ on right, adding and by using the identity $KK^*+LL^*=I_r$ we obtain
\begin{equation}\label{eq17}
T_1T^*_1+T_2T^*_2=I_r.
\end{equation}
Similarly, multiplying \eqref{eq15} on right by $K^*$ and \eqref{eq16} by $L^*$ on right and adding we have
\begin{equation}\label{eq18}
T_3T^*_1+T_4T^*_2=0.
\end{equation}
Since $\Ra(A)=\Ra(TA^*)$  implies $AA^\dag TA^*=TA^*$, from Theorem \ref{HS representation} we obtain
\begin{align*}
AA^\dag TA^*=TA^* ~ & \Leftrightarrow   \begin{bmatrix}
  T_1(\Sigma K)^*+ T_2 (\Sigma L)^* & 0 \\
  0 & 0 \end{bmatrix} = \begin{bmatrix}
  T_1(\Sigma K)^*+ T_2 (\Sigma L)^* & 0 \\
  T_3(\Sigma K)^*+ T_4 (\Sigma L)^* & 0
  \end{bmatrix},
\end{align*}

whence
\begin{equation}\label{eq19}
\Sigma K T^*_3+\Sigma L T^*_4=0.
\end{equation}
Note that  \eqref{eq11} can be rewritten as
\begin{eqnarray*}
\Sigma K &=& \Sigma K(T^*_1T_1+T^*_3T_3)+ \Sigma L (T^*_2 T_1+T_4^*T_3) \\
&=& \Sigma K T^*_1T_1+ \Sigma K T^*_3T_3+ \Sigma L T^*_2 T_1+ \Sigma L T_4^*T_3 \\
&=& (\Sigma K T^*_3T_3+ \Sigma L T_4^*) T_3+ (\Sigma K T^*_1T_1+ \Sigma L T^*_2 T_1).
\end{eqnarray*}

Thus, by \eqref{eq19} we have
\begin{equation}\label{eq20}
\Sigma K= \Sigma K T^*_1T_1+ \Sigma L T^*_2 T_1.
\end{equation}
Similarly, from \eqref{eq11} and \eqref{eq19}, we have
\begin{equation}\label{eq21}
\Sigma L=\Sigma KT^*_1T_2+\Sigma LT^*_2T_2.
\end{equation}
Multiplying \eqref{eq20} by $K^*$ on right and \eqref{eq21} by $L^*$ and adding and rearranging, we have
\begin{equation}\label{nonsingular}
I_r=(KT^*_1+LT^*_2)(T_1K^*+T_2L^*),
\end{equation}
because $\Sigma$ is nonsingular and $KK^*+LL^*=I_r$. \\
Now, we define the matrix $D:=\Sigma KT^*_1+ \Sigma LT^*_2$, which is nonsingular by \eqref{nonsingular} and the nonsingularity of $\Sigma$. By using \eqref{eq19}, we have
\begin{equation}\label{AT*}
AT^*=U\begin{bmatrix}
          D & 0 \\
          0 & 0 \\
        \end{bmatrix}U^*.
\end{equation}
Since $A=AT^*T$ from \eqref{eq17}, \eqref{eq18} and \eqref{AT*}, we have
\[A=U\begin{bmatrix}
   D & 0 \\
  0 & 0 \\
\end{bmatrix}U^*T,\]
where $T=U\begin{bmatrix}
         T_1 & T_2 \\
          T_3 & T_4 \\
        \end{bmatrix}U^*$ such that $T_1T^*_1+T_2T^*_2=I_r$ and $T_3T^*_1+T_4T^*_2=0$ hold.  \\
The proofs of (ii) $\Rightarrow $ (iii) and (iii) $\Rightarrow $ (i) are similar to  (ii) $\Rightarrow $ (iii) and (iii) $\Rightarrow $ (i) of Theorem \ref{canonical form rect}.
Finally, we suppose that  any one of equivalent conditions (i)-(iii) holds. Then,
\[CTT^*= U\begin{bmatrix}
             D & 0 \\
             0 & 0
             \end{bmatrix}
           \begin{bmatrix}
              I_r & 0 \\
                0 & Z
           \end{bmatrix}U^*=U\begin{bmatrix}
          D & 0 \\
          0 & 0 \\
        \end{bmatrix}U^*=C.\]
This completes the proof.
\end{proof}
\begin{remark} Note that the matrix $TT^*$ in (ii) $\Rightarrow $ (iii) takes the form $TT^*=U\begin{bmatrix}
 I_r & 0 \\
 0 & Z \\
 \end{bmatrix}U^*,$ where $Z=T_3T^*_3+T_4T^*_4$ is actually a hermitian matrix.
\end{remark}

The following theorem was proved by Pearl \cite {Pearl} and can be deduced from Theorem \ref{canonical form} when $T=I_n$.
\begin{corollary} Let $A\in \Cnn$ and $\rank(A)=r$. Then $A$ is EP if and only if there exists a unitary matrix $U \in \Cnn,$ and  a nonsingular matrix $D$ of order $r$ such that
       \[A= U\begin{bmatrix}
               D & 0 \\
               0 & 0 \\
              \end{bmatrix} U^*.\]
\end{corollary}

\section{More properties of square $T$-EP matrices}

As said earlier, for this section again we  take all matrices as square matrices. We obtain several characterization of $T$-EP matrices when the partial isometry $T$ satisfies some additional conditions. More precisely, when $T$ is either an orthogonal projector or a unitary matrix or a normal matrix.

We have seen in Example \ref{E2} that a EP matrix is always relative EP respect to the identity matrix of same order.  This motivates the following question: When can an EP matrix be a $T$-EP matrix  with respect to a nontrivial partial  isometry $T$?

The following results show the relationship between EP and $T$-EP matrices.

\begin{theorem}\label{AA* TEP} Let $A,T \in \Cnn.$ If $A$ is EP and $AA^*$ is $T$-EP. Then $A$ is $T$-EP.
 \end{theorem}
 \begin{proof} Since $AA^*$ is $T$-EP, by definition we have $\Ra(AA^*)=\Ra(T(AA^*)^*T)=\Ra(TAA^*T)$ and $\Nu(T)\subseteq \Nu(AA^*)$. Thus,
\begin{equation}\label{equation 1}
\Nu(T)\subseteq \Nu(AA^*)=\Nu(A^*)=\Nu(A),
\end{equation}
where the last equality is due the fact that $A$ is EP.
Also, note that
\[\Ra(A)=\Ra(AA^*)=\Ra(TAA^*T)\subseteq \Ra(TA)= T\Ra(A)=T\Ra(A^*)=\Ra(TA^*),\] and $\rk(TA^*)\le \rk(A^*)=\rk(A)$.
 Therefore,
 $\Ra(A)=\Ra(TA^*)$. In consequence, from \eqref{equation 1} and  Theorem \ref{characterization 1} (iv) we have  $A$ is $T$-EP.
\end{proof}

\begin{remark}{\rm Notice that the condition $AA^*$ is $T$-EP in above theorem can be replaced by $AA^\dag$ is $T$-EP. In fact, $\Nu(AA^*)=\Nu(AA^\dag)=\Nu(A^\dag)=\Nu(A^*)$.}
\end{remark}

Since every  unitary matrix is also a partial isometry we have the following:

\begin{theorem}\label{T unitary} Let $A,T \in \Cnn$. Suppose $T$ is an unitary matrix. Then the following are equivalent:
\begin{enumerate}[\rm (i)]
\item $A$ is $T$-EP;
\item $\Ra(A)=\Ra(TA^*)$;
\item $\Nu(A^*)=\Nu(AT^*)$;
\item $AT^*$ is EP;
\item $TA^*$ is EP;
\item $TA^\dag$ is EP;
\item $TA^\dag A=AA^\dag T$.
\end{enumerate}
In particular, under any one of equivalent conditions (i)-(vii) we have
\begin{equation}\label{T unitary 2}
(a)~ (AT)^\dag=T^* A^\dag, \quad (b)~ (TA)^\dag=A^\dag T^*,\quad\text{and}\quad (c)~(TAT)^\dag=T^*A^\dag T^*.
\end{equation}
\end{theorem}
\begin{proof} As $T$ is unitary, the first three equivalences follow from Theorem \ref{characterization 1}. Similarly,
equivalences (i)$\Leftrightarrow$(vi)$\Leftrightarrow$(vii)$\Leftrightarrow$(viii) follow from Theorem \ref{characterization 1 bis}. Also, by Theorem \ref{characterization 2} it is clear (i)$\Leftrightarrow$(vii) holds.
Finally, the expressions in \eqref{T unitary 2} can be easily proved by mean a direct verification of the definition of Moore-Penrose inverse and by applying again the fact that $TT^*=T^*T=I_n$.
\end{proof}

Recall that an involutory matrix is a nonsingular matrix that is its own inverse. Next, we present interesting characterizations of $T$-EP matrices when $T$ is an  involutory hermitian matrix, that is, $T^{-1}=T=T^*$.
\begin{theorem}\label{T hermitian involutory} Let $A,T \in \Cnn$. Suppose $T$ is an involutory hermitian matrix. Then the following are equivalent:
\begin{enumerate}[\rm (i)]
\item $A$ is $T$-EP;
\item $\Ra(A)=\Ra(TA^*)$;
\item $\Nu(A^*)=\Nu(AT)$;
\item $AT$ is EP;
\item $TA^*$ is EP;
\item $TA^\dag$ is EP;
\item $TA^\dag A=AA^\dag T$;
\item $\Ra(A^*)=\Ra(TA)$;
\item $\Nu(A)=\Nu(A^* T)$;
\item $A^*T$ is EP;
\item $TA$ is EP;
\item $A^\dag T$ is EP;
\item $A^\dag A T= TAA^\dag$.
\end{enumerate}
\end{theorem}
\begin{proof} As $T$ satisfies $T^{-1}=T=T^*$, equivalences (i) to (xii) follow from Theorems \ref{characterization 4} and \ref{T unitary}.\\
(iv)$\Leftrightarrow$(xiii). It is clear that $(AT)^\dag=TA^\dag$ because $T^{-1}=T=T^*$. Moreover, $T^2=I_n$. Thus, $AT$ is EP if and only if
\begin{eqnarray*}
AT(AT)^\dag &=& (AT)^\dag AT \\
AT TA^\dag &=& TA^\dag AT \\
AA^\dag &=& TA^\dag AT \\
T AA^\dag &=& A^\dag AT.
\end{eqnarray*}
\end{proof}

\begin{theorem}\label{T orthogonal projector} Let $A,T \in \Cnn.$ Suppose $T$ is an orthogonal projector.  Then the following are equivalent:
\begin{enumerate}[\rm (i)]
\item $A$ is $T$-EP;
\item $A$ is EP and $A=AT$;
\item $A^*$ is EP and $A=AT$;
\item $A^\dag$ is EP and $A=AT$.
\end{enumerate}
In particular, under any one of equivalent conditions (i)-(iv) we have that $AT$, $TA$, and also $TAT$ are all EP and $T$-EP.
\end{theorem}
\begin{proof}  Note that if $T$ is an orthogonal projector,  $A=AT^*T$ if and only if $A=AT$. Now, all equivalences follow as a direct application of Theorem \ref{characterization 1 bis}. \\
In order to prove the last affirmation we suppose (ii) holds. Clearly, $A=AT$. On the other hand, by Theorem \ref{properties 1} (i) we know that $A=TT^*A$ or equivalently $A=TA$, as $T$ is an orthogonal projector. Consequently, $A=AT=TA$ and so $A=TAT$. Therefore, part (ii) implies $A$ is $T$-EP and therefore $AT$, $TA$, and also $TAT$ are all EP and $T$-EP.
\end{proof}

\begin{lemma}\label{lemma T normal} Let $A,T \in \Cnn.$ Let $T$ be a normal partial isometry. Then the following hold:
\begin{enumerate}[{\rm(i)}]
\item If $A=TT^*A$ then $(TA)^\dag=A^\dag T^*$. In particular, $A^\dag= (TA)^\dag T$.
\item If $A=AT^*T$ then $(AT)^\dag= T^* A^\dag$. In particular, $A^\dag= T (AT)^\dag$.
\end{enumerate}
\end{lemma}
\begin{proof} (i) We check this by direct verification. Let $X:=A^\dag T^*$. Since $A=TT^*A$ and $T$ is normal (i.e., $TT^*=T^*T$) we have,
\[TAXTA=TA A^\dag T^*TA = TA A^\dag TT^*A=TA A^\dag A=TA,\]
\[XTAX=A^\dag T^* TA A^\dag T^*= A^\dag TT^*AA^\dag T^*=A^\dag AA^\dag T^*=A^\dag T^*=X,\]
\[(TAX)^*=(TAA^\dag T^*)^*=T(AA^\dag)^* T^*=TAA^\dag T^*=TAX,\]
\[(XTA)^*=(A^\dag T^*TA)^*=(A^\dag TT^*A)^*= (A^\dag A)^*=A^\dag A=A^\dag TT^*A=A^\dag T^* TA= XTA.\]
By uniqueness of the Moore-Penrose inverse we have $X=(TA)^\dag$. \\
It remains to show the last affirmation of part (i). In fact, from Lemma \ref{lemma 2} (i) and the fact that $T$ is normal we obtain $A^\dag=A^\dag T T^*=A^\dag T^* T=(TA)^\dag T$.  \\
(ii) Can be proved similarly.
\end{proof}

\begin{theorem} Let $A,T \in \Cnn.$ Let $T$ be a normal partial isometry. If $A$ is $T$-EP then the Moore-Penrose inverse of $A$ is given by \[A^\dag=(TA)^\dag T=T (AT)^\dag.\]
\end{theorem}
\begin{proof}
As $A$ is $T$-EP we know that $A=AT^*T$ by definition. Also, from Theorem \ref{properties 1} (i) we have $A=TT^*A$. Now, the expression of $A^\dag$ follows from Lemma  \ref{lemma T normal}.
\end{proof}

\begin{theorem}\label{square case 8} Let $A,T \in \Cnn.$ Let $T$ be an hermitian partial isometry.  Then $A$ is $T$-EP if and only if $TA$ is EP and $A=AT^2=T^2A$.
 \end{theorem}
\begin{proof} Let $A$ be $T$-EP. Then by definition of $T$-EP matrix  and Theorem \ref{properties 1} (i), respectively, we have  $A=AT^*T$ and $A=TT^*A$ which are equivalent to $A=AT^2=T^2A$ as $T$ is hermitian. \\
In order to show that $TA$ is EP is equivalent we will use the well-known characterization $TA(TA)^\dag=(TA)^\dag TA$. In fact, first we note that $T$ is a normal partial isometry. Thus, as $A=TT^*A$, Lemma \ref{lemma T normal} (i) implies $(TA)^\dag=A^\dag T^*=A^\dag T$. Moreover, $(TA)^\dag TA = A^\dag A$.  Also, Lemma \ref{lemma 1} (iv) and Theorem \ref{characterization 2} imply $A^\dag=T^*TA^\dag=T^2A^\dag$  and $AA^\dag T=TA^\dag A$, respectively. Therefore,
\begin{eqnarray*}
TA(TA)^\dag &=& TAA^\dag T \\
&=& T T A^\dag A \\
&=& A^\dag A \\
&=& (TA)^\dag TA.
\end{eqnarray*}
Conversely, let $TA$ be EP  and $A=T^2A=AT^2$. Clearly, $A=TT^*A$ and $T$ is a normal partial isometry. Thus, by Lemma \ref{lemma T normal} (i) we have $(TA)^\dag=A^\dag T^*=A^\dag T$.  Therefore, $TA(TA)^\dag=(TA)^\dag TA$ is equivalent to   $TAA^\dag T=A^\dag T^2 A=A^\dag A$. Multiplying on left this equality, by $T$ and using $A=AT^*T=T^2A=AT^2$  implies $AA^\dag T=TA^\dag A$. Thus, Theorem \ref{characterization 2} completes the proof.
\end{proof}

\begin{theorem}\label{square case 3} Let $A,T \in \Cnn,$ and $T$ a partial isometry such that $A$ and $T^*$ commute. Then $A$ is $T$-EP if and only if $A$ is EP and $A=AT^*T.$
\end{theorem}
\begin{proof} Suppose $A$ and $T^*$ commute and $A$ is $T$-EP. By definition, $\Ra(A)=\Ra(TA^*T).$ Since $A$ and $T^*$ commute, $A^*$ and $T$ commute. Therefore $\Ra(A)= \Ra(A^*T^2)\subseteq \Ra(A^*).$ However, $\rank(A)=\rank(A^*)$ implies $\Ra(A)=\Ra(A^*).$ Thus $A$ is EP.\\
Conversely, let $A$ be $EP$ and $A=AT^*T.$ As $A$ is EP, $\Ra(A)=\Ra(A^*).$ Thus, \[\Ra(TA^*)=\Ra(AT^*)\subseteq\Ra(A^*)=\Ra(A).\] As  $A=AT^*T,$ we obtain
$\rank(A)=\rank(AT^*T)\leq \rank(AT^*)= \rank(TA^*)$. So, $\Ra(TA^*)=\Ra(A).$ Now, by Theorem 2.13 (v),  $A$ is $T$-EP.
\end{proof}

\section{Sums of $T$-EP matrices}
We now study the problem: When is sum of two $T$-EP matrices  a $T$-EP matrix? This problem is not completely resolved even in case of EP matrices and we intend to explore it for $T$-EP matrices in this final section.

\begin{theorem}\label{new result 5} Let $A,B \in \Cm$ be $T$-EP  such that $A^*B+B^*A=0$. Then $A+B$ is $T$-EP.
\end{theorem}
\begin{proof} Let $A$ and $B$ be $T$-EP. As $A^*B+B^*A=0$  we have $TA^*BT^*+TB^*AT^*=0$, and so
\begin{align*}
(AT^*+BT^*)^*(AT^*+BT^*) &=(TA^*+TB^*)(AT^*+BT^*) \\
&=TA^*AT^*+TA^*BT^*+TB^*AT^*+TB^*BT^*\\
&=TA^*AT^*+TB^*BT^*.
\end{align*}
Thus,
\begin{eqnarray*}
\Nu(AT^*+BT^*)&=&\Nu((AT^*+BT^*)^*(AT^*+BT^*))\\
        &=& \Nu(TA^*AT^*+TB^*BT^*) \\
        &=& \Nu\left(\begin{bmatrix} AT^*\\BT^*
\end{bmatrix}^*\begin{bmatrix} AT^*\\BT^*
\end{bmatrix}\right) \\
        &=& \Nu\begin{bmatrix} AT^*\\BT^*
\end{bmatrix} \\
        &=& \Nu(AT^*)\cap \Nu(BT^*).
\end{eqnarray*}
From Theorem \ref{characterization 1 bis} we know that $AT^*$ is EP and $A=AT^*T$, and also $BT^*$ is EP and $B=BT^*T$.    Therefore,
\begin{align*}
\Nu(AT^*+BT^*)& = \Nu(AT^*)\cap\Nu(BT^*)\\
&=\Nu(TA^*)\cap\Nu(TB^*)\\
& \subseteq \Nu(TA^*+TB^*)\\
&=\Nu(T(A+B)^*) \\
&= \Nu((AT^*+BT^*)^*),
\end{align*}
whence $\Nu(AT^*+BT^*)=\Nu((AT^*+BT^*)^*)$.  Thus,  $(A+B)T^*$ is EP. Also, it is clear that $A+B=(A+B)T^*T$. Now, the result follows from Theorem \ref{characterization 1 bis}.
\end{proof}

\begin{corollary}\label{A*B=0} Let $A,B \in \Cm$ be $T$-EP  such that  $A^*B=0$. Then,  $A+B$ is $T$-EP.
 \end{corollary}
 \begin{proof}
It is a direct consequence from Theorem \ref{new result 5}.
 \end{proof}
 
\begin{corollary}\label{BA*=0} Let $A,B \in \Cm$ be $T$-EP  such that  $BA^*=0$. Then,  $A+B$ is $T$-EP.
 \end{corollary}
\begin{proof}
Firstly note that $BA^*=0$ can be written as $(B^*)^*A^*=0$. Also, by Theorem \ref{characterization 4} we know that $A^*$ and $B^*$ are $T^*$-EP. Thus, from Corollary \ref{A*B=0} we have that $B^*+A^*$ is $T^*$-EP. Thus, by applying  again Theorem \ref{characterization 4} we get $A+B$ is $T$-EP.
\end{proof}

An interesting consequence of previous corollaries is the following theorem. Before we need tot define the concept of $*$-orthogonality \cite{Hes} of two matrices.
 \begin{definition}\label{*orthogonal} Let $A,~B \in \Cm.$ We say $A$ and $B$ are $*$-orthogonal in case $A^*B=0$ and $BA^*=0$.
 \end{definition}

It is easy to see that $A$ and $B$ are $*$-orthogonal if and only if $A^*$ and $B^*$ are $*$-orthogonal. Moreover, if $A$ and $B$ are both $T$-EP, then
  (i) $A$ and $B$ are $*$-orthogonal implies $AT^*$ and $BT^*$ are $*$-orthogonal and (ii) $TA^*$ and $TB^*$ are $*$-orthogonal.

\begin{theorem}\label{sum *orthogonal} Let $A,B \in \Cm$ be $T$-EP. If $A$ and $B$ are $*$-orthogonal, then $A+B$ is $T$-EP.
\end{theorem}
\begin{proof}  Follows by Definition \ref{*orthogonal} and Corollary \ref{A*B=0}.
\end{proof}

In the following two results we present another sufficient  conditions for a sum of two $T$-EP matrices to be $T$-EP. 

\begin{theorem}\label{sum1} Let $A,B \in \Cm$ be $T$-EP  such that  $\Ra(A) \cap \Ra(B)=\{0\}$. Then, $A+B$ is $T$-EP.
 \end{theorem}
 \begin{proof} Since $A$ and $B$ are $T$-EP, by  Theorem \ref{characterization 1 bis} we have $AT^*$ is EP and $A=AT^*T$, and also $BT^*$ is EP and $B=BT^*T.$ We will show that $(A+B)T^*$ is EP and $(A+B)=(A+B)T^*T$. Clearly, the second condition is true. On the other hand,  note that
\begin{equation}\label{kernel}
\Nu\left(\begin{bmatrix}
           AT^* \\
          BT^*
\end{bmatrix}\right)=\Nu(AT^*)\cap \Nu(BT^*)\subseteq \Nu(AT^*+BT^*).
\end{equation}
Also, we have $\Ra(AT^*) \cap \Ra(BT^*)\subseteq \Ra(A) \cap \Ra(B)=\{0\}$. In consequence, as $AT^*$ and $BT^*$ are EP, also we have  $\Ra(TA^*) \cap \Ra(TB^*)=\{0\}$. Thus,
\begin{equation}\label{rank sum}
\rank(AT^*+BT^*)=\rank(AT^*)+\rank(BT^*)=\rank\begin{bmatrix}
           AT^* \\
          BT^*
\end{bmatrix}.
\end{equation}
Therefore, from \eqref{kernel} and \eqref{rank sum} we obtain $\Nu(AT^*+BT^*)=\Nu((AT^*+BT^*)^*)$, i.e., $(A+B)T^*$ is EP.  Once again by Theorem \ref{characterization 1 bis}, $A+B$ is $T$-EP.
 \end{proof}

\begin{theorem}\label{sum2} Let $A,B \in \Cm$ be $T$-EP  such that  $\Ra(A^*) \cap \Ra(B^*)=\{0\}$. Then, $A+B$ is $T$-EP.
 \end{theorem}
\begin{proof}
Let $A$ and $B$ be $T$-EP. From Theorem \ref{characterization 4} we know that $A^*$ and $B^*$ are $T^*$-EP. Thus, Theorem \ref{sum1} implies $A^*+B^*$ is $T^*$-EP. Once again Theorem \ref{characterization 4} we obtain that $A+B$ is $T$-EP.
\end{proof}

We end this section with yet another sufficient conditions for a sum of two $T$-EP matrices to be $T$-EP. 

 \begin{theorem}\label{sum4} Let $A,B \in \Cm$ be $T$-EP  such that $\rank(T(A+B)^*T)=\rank(A+B)$ and  $\Ra(A+B)=\Ra(A)+\Ra(B)$. Then $A+B$ is $T$-EP.
\end{theorem}
\begin{proof} As $A$ and $B$ are $T$-EP, by definition  we have $\Ra(A)=\Ra(TA^*T)$ and $A=AT^*T$, and also $\Ra(B)=\Ra(TB^*T)$ and $B=BT^*T$. In consequence, from $\Ra(A+B)=\Ra(A)+\Ra(B)$ we obtain  \[\Ra(T(A+B)^*T) \subseteq \Ra(TA^*T)+\Ra(TB^*T)=\Ra(A)+\Ra(B)=\Ra(A+B).\]
Thus, $\rank(T(A+B)^*T)=\rank(A+B)$ implies $\Ra(T(A+B)^*T) = \Ra(A+B).$ Moreover, it is easy to see $(A+B)=(A+B)T^*T.$ Hence $A+B$ is $T$-EP.
\end{proof}

\begin{remark}\label{last remark}{\rm Note that if $A$ and $B$ are $T$-EP and $*$-orthogonal then (i)~ $\rank(T(A+B)^*T)=\rank(A+B)$ and (ii)~ $\Ra(A+B)=\Ra(A)+\Ra(B)$. In fact, (i) it follows from Corollary \ref{sum *orthogonal} and Theorem \ref{properties 1} (ii). \\
(ii) It is a consequence of the fact that $A^*B=0$ and $BA^*=0$ imply $\Ra(A) \cap \Ra(B)=\{0\}$ and $\Ra(A^*) \cap \Ra(B^*)=\{0\}$, respectively. Thus, $\rank(A+B)=\rank(A)+\rank(B)$ and therefore $\Ra(A+B)=\Ra(A)+\Ra(B)$.
}
\end{remark}

\end{document}